\documentclass{amsart} 
\usepackage{amsmath,graphics,amsfonts,amsthm}
\usepackage{graphicx, amsfonts, amsmath, amsthm, amssymb, amscd, tikz, tipa, hyperref, verbatim}
\usepackage{enumerate}
\usepackage{wasysym}
\usepackage{color}
\usepackage{ytableau}

\usetikzlibrary{decorations.markings}
\tikzset{node distance=2.5cm, auto}

\author[Ganguly]{Jyotirmoy Ganguly}
 
\author[Spallone]{Steven Spallone}

\address{Indian Institute of Science Education and Research, Pune-411008, India}
\email{jyotirmay.ganguli@students.iiserpune.ac.in}
\address{Indian Institute of Science Education and Research, Pune-411008, India}
\email{sspallone@gmail.com}



\newcommand{\nc}{\newcommand}
\nc{\dmo}{\DeclareMathOperator}

\nc{\mc}{\mathcal}
\nc{\mb}{\mathbb}
\nc{\mf}{\mathfrak}
\nc{\ul}{\underline}
\nc{\ol}{\overline}
\nc{\lip}{\langle}
 \nc{\rip}{\rangle}

\nc{\N}{\mb N}
\nc{\R}{\mb R}
\nc{\Z}{\mb Z}
\nc{\Q}{\mb Q}
\nc{\C}{\mb C}

\nc{\beq}{\begin{equation*}}
\nc{\eeq}{\end{equation*}}
\nc{\half}{\frac{1}{2}}
\nc{\la}{\lambda}
\nc{\eps}{\epsilon}

\DeclareMathOperator{\Spin}{Spin}
\DeclareMathOperator{\Pin}{Pin}

\DeclareMathOperator{\GL}{GL}
\dmo{\SL}{SL}
\dmo{\SO}{SO}
\dmo{\Or}{O}
\dmo{\sgn}{sgn}
\dmo{\Hom}{Hom}
\dmo{\odd}{odd}
\dmo{\Sym}{Sym}
\dmo{\im}{im}

\newtheorem{theorem}{Theorem}[section]
\newtheorem{lemma}[theorem]{Lemma}

\newtheorem{prop}[theorem]{Proposition}

\newtheorem{cor}[theorem]{Corollary}

\theoremstyle{definition}

\theoremstyle{remark}

\numberwithin{equation}{section}

  \dmo{\tr}{tr}
  
\title{Spinorial Representations of Symmetric   Groups}
\makeatletter\@addtoreset{chapter}{part}\makeatother%

\begin{document}

\maketitle

\begin{abstract}
 A real representation $\pi$ of a finite group  may be regarded as a homomorphism   to an orthogonal group $\Or(V)$.  For symmetric groups $S_n$, alternating groups $A_n$, and products $S_n \times S_{n'}$ of symmetric groups, we give criteria for whether $\pi$ lifts to the double cover $\Pin(V)$ of $\Or(V)$, in terms of character values.  From these criteria we compute the second Stiefel-Whitney classes of these representations. 
\end{abstract}

\tableofcontents

\section{Introduction}

\bigskip

A real representation $\pi$ of a finite group $G$ can be viewed as a group homomorphism from $G$ to the orthogonal group  $\Or(V)$ of a Euclidean space $V$.  Recall the double cover $\rho: \Pin(V) \to \Or(V)$.  We say that  $\pi$ is {\it spinorial}, provided it lifts to $\Pin(V)$, meaning there is a homomorphism $\hat \pi: G \to \Pin(V)$ so that $\rho \circ \hat \pi=\pi$.   

When the image of $\pi$ lands in $\SO(V)$, the representation is spinorial precisely when its second Stiefel-Whitney class $w_2(\pi)$ vanishes.  Equivalently, when the associated vector bundle over the classifying space  $BG$ has a spin structure.   (See Section 2.6 of \cite{Benson}, \cite{GKT}, and Theorem II.1.7 in \cite{Lawson}.)  Determining spinoriality of Galois representations also has applications in number theory:  see \cite{Serrespin}, \cite{Deligne}, and    \cite{DPDR}.

 In this paper we give lifting criteria for representations of the symmetric groups $S_n$, the alternating groups $A_n$, and a product $S_{n} \times S_{n'}$ of two symmetric groups.
Write $s_i \in S_n$ for the transposition $(i,i+1)$, in cycle notation. A key result of this paper is the following:

\begin{theorem} \label{char.spin.crit} Let $n \geq 4$.
\begin{enumerate}
\item A representation $\pi$ of $S_n$ is spinorial iff $ \chi_\pi(1) \equiv \chi_\pi(s_1s_3)  \mod 8$ and $\chi_\pi(1)-\chi_\pi(s_1)$
is congruent  to $0$ or $6 \mod 8$.  

\item  A representation $\pi$ of $A_n$ is spinorial iff  $\chi_\pi(1) \equiv \chi_\pi(s_1s_3)   \mod 8$. 
\end{enumerate}
\end{theorem}

Combining this with the main result of \cite{GPS} on character values, one deduces that as $n \to \infty$, ``$100 \%$" irreducible representations of $S_n$ are spinorial.  (See Corollary \ref{percent}.) 

Next, we leverage this result to compute the second Stiefel-Whitney classes for (real) representations $\pi$ of $S_n$:
\begin{equation} \label{w2.is.here}
w_2(\pi)  =\left[\frac{\chi_\pi(1)-\chi_\pi(s_1)}{4}\right]e_{\mathrm{cup}}+\frac{\chi_{\pi}(1)-\chi_\pi(s_1s_3)}{4}w_2(\pi_n),
\end{equation}
where $\pi_n$ is the standard representation of $S_n$, and $e_{\mathrm{cup}} \in H^2(G,\Z/2\Z)$ is a certain cup product.

This formula allows us to compute the second Stiefel-Whitney classes of representations of   $S_n \times S_{n'}$ through K\"{u}nneth theory, and therefore to identify spinorial representations of this product.  To state the result, let $\Pi=\pi \boxtimes \pi'$ be the external tensor product of representations $\pi$ of $S_n$ and $\pi'$ of $S_{n'}$.     Let $g=\half(\chi_\pi(1)-\chi_\pi(s_1))$, the multiplicity of $-1$ as an eigenvalue of $\pi(s_1)$ and similarly write $g'$ for the corresponding quantity for $\pi'$.

  \begin{theorem} \label{prod.thm.intro} The representation $\Pi$ of $S_n \times S_{n'}$ is spinorial iff the restrictions of $\Pi$ to $S_n \times \{ 1 \}$ and $\{ 1 \} \times S_{n'}$ are spinorial, and  \beq
   (\deg \Pi + 1)gg' \equiv 0 \mod 2.
  \eeq
  \end{theorem}

We now describe the layout of this paper.  Section \ref{preliminary.section} reviews the group $\Pin(V)$ and other conventions.  The $S_n$ case of Theorem \ref{char.spin.crit} is proven in Section \ref{key.section.sn} by means of defining relations for the $s_i$.  Additionally   we note corollaries of Theorem \ref{char.spin.crit}: primarily  the aforementioned ``$100 \%$" result,  a connection with skew Young tableau numbers, and the important case of  permutation modules, meaning the induction of the trivial character from a Young subgroup of $S_n$.   In particular we demonstrate that the regular representation of $S_n$ is spinorial for $n \geq 4$.
 
 Representations of the alternating groups are treated in Section \ref{altgrp}, again via generators and relations.   The main result is the $A_n$ case of Theorem \ref{char.spin.crit}. We enumerate the spinorial irreducible representations of $A_n$ in Theorem  \ref{irred.spin.an}.  Data for spinoriality of irreducible representations of $S_n$ and $A_n$ for small $n$ is presented in Tables 1 and 2 of Section \ref{table.section}.

In Section \ref{SW.section} we review the axioms of Stiefel-Whitney classes of real representations, and then deduce the Stiefel-Whitney class of a real representation of $S_n$.  In Section  \ref{products.section} we apply K\"{u}nneth theory to this formula to compute Stiefel-Whitney classes for real representations of $S_n \times S_{n'}$.  From this it is straightforward to deduce Theorem \ref{prod.thm.intro}.

 \section*{Acknowledgements}
 The authors would like to thank Dipendra Prasad, Rohit Joshi, and Amritanshu Prasad for helpful conversations.  The research by the first author for this article was supported by a PhD scholarship from the Indian National Board for Higher Mathematics.

  \section{Notation and Preliminaries} \label{preliminary.section}
  
  \subsection{Representations}
  All representations are on finite-dimensional vector spaces, which are always real, except in Section  \ref{altgrp}, where they may be specified as complex.   For a representation $(\pi,V)$  of a group $G$, write `$\det \pi$' for the composition $\det \circ \pi$; it is a linear character of $G$.  Also write `$\chi_\pi$' for the character of $\pi$.  If $H \leq G$ is a subgroup, write $\pi|_H$ for the restriction of $\pi$ to $H$.    A real representation $\pi: G \to \GL(V)$ can be conjugated to have image in $\Or(V)$, so we will assume that this is the case.
  When $\det \pi$ is trivial, it maps to $\SO(V)$, and the spinoriality question is whether it lifts to the double cover $\Spin(V)$ (which we review in the next section).

 Let $\sgn: S_n \to \{ \pm 1 \}$ be the usual sign character.  For $G=S_n$, we say that $\pi$ is \emph{chiral} provided $\det \pi=\sgn$ and $\pi$ is \emph{achiral} provided $\det \pi=1$.      
   Write $\pi_n: S_n \to \GL_n(\R)$ for the standard representation of $S_n$ by permutation matrices.

\subsection{The Pin Group}
We essentially review \cite[Chapter $1.6$]{BrokerDieck} for defining the groups $\Spin(V)$ and $\Pin(V)$, where $V$ is a Euclidean (i.e., a normed finite-dimensional real vector) space.
  The Clifford algebra $C(V)$ is the quotient of the tensor algebra $T(V)$ by the two-sided ideal generated by the set
\beq
 \{ v\otimes v +|v|^2 : v\in V\}.
 \eeq
 Write $C(V)^\times$ for its group of units.

We identify $V$ as a subspace of $C(V)$ through the natural injection $i: V \to C(V)$.
  Write $\alpha$ for the unique involution of the $\R$-algebra $C(V)$ with the property that $\alpha(x)=-x$ for $x \in V$. 
  One has
  \beq
  C(V) = C(V)^0 \oplus C(V)^1,
  \eeq
  where $C(V)^0$ is the $1$-eigenspace of $\alpha$ and $C(V)^1$ is the $-1$-eigenspace.
  
   Write $t$ for the unique anti-involution of $C(V)$ with $t(x)=x$.  For $x \in C(V)$, define $\ol x=t(\alpha(x))$; it is again an algebra anti-involution.  Define
 \beq
 N: C(V) \to C(V)
 \eeq
 by $N(x)= x \ol x$.
 Put
 \beq
 \Gamma_V=\{ x \in C(V)^\times \mid \alpha(x) V x^{-1}= V   \}.
 \eeq
Let  $\rho: \Gamma_V \to \GL(V)$ be the homomorphism   given by $v \mapsto \alpha(x) v x^{-1}$.  
 We will repeatedly use the fact that if $v$ is a unit vector, then $\rho(v)$ is the reflection determined by $v$.  Write `$\Pin(V)$' for the kernel of the restriction of $N$ to $\Gamma_V$.  The restriction of $\rho$ to $\Pin(V)$ is a double cover of $\Or(V)$ with kernel $\{ \pm 1 \}$.
The preimage of $\SO(V)$ under $\rho$ is denoted `$\Spin(V)$'.  Alternately, $\Spin(V)=\Pin(V) \cap C(V)^0$.
 
  \section{Symmetric Groups} \label{key.section.sn}
 
 \subsection{Lifting Criteria} \label{lifting.criteria.subs}
 Let $n \geq 2$.  The group $S_n$ is generated by the transpositions $s_i=(i,i+1)$ for $1 \leq i \leq n-1$,
with the following relations:

\begin{enumerate}
 \item 
     ${s_i}^2=1,\quad 1\leq i\leq n-1$,
 \item
     $s_is_k=s_k s_i, \text{ when } \quad |i-k|>1$,
      \item
     ${(s_is_{i+1})}^3=1,\quad 1\leq i\leq n-2$.

\end{enumerate}

 Therefore, defining a homomorphism from $S_n$ to a group $G$ is equivalent to choosing elements $x_1, \ldots, x_{n-1} \in G$ satisfying the same relations.  Let us call the relation $x_i^2=1$  the ``first lifting condition", the relation $x_ix_k=x_k x_i$ the ``second lifting condition", and
  $(x_i x_{i+1})^3=1$  the ``third lifting condition".  Note that this second condition is vacuous for $n < 4$.
 
 Let $\pi: S_n \to \Or(V)$ be a representation of degree $d$.  For each $\pi(s_i) \in \Or(V)$ there are $ \pm c_i \in \Pin(V)$ with $\rho(\pm c_i)=\pi(s_i)$, and the question is whether we may choose signs so that the $x_i=\pm c_i$ satisfy these lifting conditions.

 Let $g_\pi=\dfrac{\chi_\pi(1)-\chi_\pi(s_1)}{2}$, as in \cite{APS}.  This is the multiplicity of the eigenvalue $-1$ of $\pi(s_1)$, and the eigenvalue $1$ occurs with multiplicity $d-g_\pi$.  
 Put  $c_i=u_1 \cdots u_{g_\pi} \in \Pin(V)$, where $u_1, \ldots, u_{g_\pi}$ is an orthonormal basis of the $-1$-eigenspace of $\pi(s_i)$.
  Since $\pi(s_i)$ is the product of the reflections in each $u_j$, the elements $c_i$ and $-c_i$ are the lifts of $\pi(s_i)$. One computes that 
 \beq
 c_i^2=(-c_i)^2=(-1)^{\half g_\pi(g_\pi +1)},
 \eeq
 and therefore the first lifting condition is satisfied iff $g_\pi$ is congruent to $0$ or $3$ modulo $4$.  It does not matter for this whether we choose $c_i$ or $-c_i$.  
 
Consider the sequence $(c_1c_2)^3, (c_2 c_3)^3, \ldots \in \Pin(V)$.  Since each $( \pi(s_i) \pi(s_{i+1}))^3=1$, this must be a sequence of $\pm 1$'s.   For the third lifting condition these must each be $1$.  Thus $c_1$ may take either sign, but then the signs for $c_2, c_3, \ldots$ are determined.  Moreover this does not affect the first lifting condition.  Thus:

\begin{prop} The first and third lifting conditions hold iff $g_\pi \equiv 0$ or $3 \mod 4$. 
\end{prop}


Now let $|i-k|>1$, and suppose as above that $c_i^2=1=c_k^2=1$.  Then the second lifting condition holds iff $(c_ic_k)^2=1$.  By conjugating we may assume that $i=1$ and $k=3$.  So put $h_\pi=\dfrac{\chi_\pi(1)-\chi_\pi(s_1s_3)}{2}$; as above the condition is equivalent to $h_\pi \equiv 0,3 \mod 4$.  However:

\begin{lemma} The integer $h_\pi$ is even.
\end{lemma}

\begin{proof} Let $\zeta_4$ be a $4$-cycle in $S_n$.  Then $\zeta_4^2$ is conjugate to $s_1s_3$.  Let $m$ be the multiplicity of $i=\sqrt{-1}$ as an eigenvalue of $\pi(\zeta_4)$.  Then $h_\pi$, the  multiplicity of $-1$ as an eigenvalue of $\pi(s_1s_3)$, is $2m$.
\end{proof}

Therefore:

\begin{prop} The second lifting condition holds iff $h_\pi$ is a multiple of $4$.
\end{prop}
\qedsymbol

We summarize the above as the following:

\begin{theorem} \label{tfae.spin} Let $n \geq 4$, and $\pi$ a representation of $S_n$.  The following are equivalent:
\begin{enumerate}
\item The representation $\pi$ is spinorial.
\item $g_\pi \equiv 0$ or $3 \mod 4$, and $h_\pi \equiv 0 \mod 4$.
 \item  $\chi_\pi(1)-\chi_\pi(s_1) \equiv 0$ or $6 \mod 8$, and
$ \chi_\pi(1) \equiv \chi_\pi(s_1s_3)  \mod 8$. \end{enumerate}
When $\pi$ is spinorial it has two lifts.
\end{theorem}

When $\pi$ is spinorial, note that $g_\pi \equiv 0$ iff $\pi$ is achiral, and $g_\pi \equiv 3$ iff $\pi$ is chiral.

\begin{proof} The equivalence should be clear; the two lifts correspond to the choice of sign for $c_1$.
\end{proof}

{\bf Remark:} The two lifts correspond to the two members of $H^1(S_n,\Z/2\Z)$; see Theorem II.1.7 of \cite{Lawson}.
  
  \begin{cor} \label{rest.spin.prop} A representation $\pi$ of $S_n$ is spinorial iff its restrictions to the cyclic subgroups  $\lip s_1 \rip$ and  $\lip s_1 s_3 \rip$ are both spinorial.
  \end{cor}
  \begin{proof} Indeed, the $g_\pi$ condition corresponds to the subgroup $\lip s_1 \rip$, and the $h_\pi$ condition corresponds to  $\lip s_1 s_3 \rip$.
  \end{proof}
   
  The irreducible representations  of $S_n$ are the Specht modules $(\sigma_\lambda,V_\lambda)$, indexed by partitions of $n$.     (See for instance \cite{jameskerber}.) Write $f_\lambda=f_{\pi_\lambda}$, and similarly for   $g_\lambda, h_\lambda$.   Write $p(n)$ for the number of partitions of $n$.
  
  \begin{cor}  \label{percent} We have
  \beq
  \lim_{n \to \infty} \frac{ \# \{ \lambda \vdash n \mid \sigma_\lambda \text{ is achiral and spinorial   } \}}{p(n)}=1.
  \eeq
  \end{cor}
  
  In other words, as $n \to \infty$, $100\%$ of irreducible representations of $S_n$ are achiral and spinorial.
  
  \begin{proof} 
  According to \cite{GPS}, as $n \to \infty$, $100\%$ of partitions $\la$ of $n$ have 
  \beq
  \chi_\la(1) \equiv \chi_\la(s_1) \equiv \chi_\la(s_1s_3) \equiv 0 \mod 8.
  \eeq
  The conclusion then follows from Theorem \ref{char.spin.crit}.

  \end{proof}

  \subsection{Connection with Skew Young Tableaux}

  Let $\mu,\la$ be partitions for which the Young diagram of $\la$ contains that of $\mu$.  The notion of standard Young tableaux generalizes to ``skew diagrams" $\la / \mu$.  Following Section 7.10 in \cite{RStan2}, write $f_{\la / \mu}$ for the number of SYT on $\la / \mu$.  (If the Young diagram of $\mu$ is not contained in that of $\la$, put $f_{\la / \mu}=0$.)   
    
  \begin{prop} We have
  \begin{enumerate}
  \item $g_\lambda=f_{\lambda / (1,1)}$ and
  \item \label{blub} $h_\lambda=2 \cdot (f_{\la / (3,1)} + f_{\la / (2,1,1)})$.
  \end{enumerate}
  \end{prop}
  
  \begin{proof} 
  
  Let $\mu \vdash k$ for some $k \leq n$, and let $\ul{\mu}$ be the partition of $n$ defined by adding $(n-k)$ $1$'s, i.e. $\ul \mu=\mu+\underbrace{1+ \cdots + 1}_{(n-k) \text{ times}}$.  Write $w_{\mu} \in S_n$ be a permutation with cycle type $\mu$.  According to \cite{RStan2}, Exercise 7.62, we have  
  \beq
  \chi_\lambda(w_{\ul \mu})=\sum_{\nu \vdash k} \chi_\nu(w_\mu) \cdot f_{\la / \nu}.
  \eeq
  
  Taking $\mu=(2)$ gives
  \beq
  \chi_\la(s_1)=f_{\la / (2)}-f_{\la / (1,1)},
  \eeq
  and taking $\mu=(1,1)$ gives
  \beq
  \chi_\la(1)=f_{\la / (2)}+f_{\la / (1,1)},
  \eeq
  so that $g_\la=f_{\la / (1,1)}$.
  
  Similarly, taking $\mu=(2,2)$ and using the character table for $S_4$, we compute
  \begin{equation} \label{common}
  \chi_\la(s_1s_3) = f_{\la/ (4)} - f_{\la / (3,1)}+ 2 f_{\la / (2,2)}-f_{\la / (2,1,1)}+ f_{\la / (1,1,1,1)}.
  \end{equation}
   Taking $\mu=(1,1,1,1)$ gives
   \begin{equation} \label{ground}
   \chi_\la(1)= f_{\la/ (4)} +3 f_{\la / (3,1)}+ 2 f_{\la / (2,2)}+3 f_{\la / (2,1,1)}+ f_{\la / (1,1,1,1)}.
   \end{equation}
   Combining \eqref{common} and \eqref{ground} gives the formula  for $h_\la$.

 
  \end{proof}
 

  \subsection{Permutation Representations} \label{perm.section}
  Another important class of representations of $S_n$ are the permutation representations, which are also indexed by partitions of $n$.
  Let $\la=(\la_1, \ldots, \la_\ell) \vdash n$, and consider the set $\mc P_\la$ of ordered partitions of $\{ 1, 2, \ldots, n \}$ with shape $\la$.  
  Thus a member of $\mc P_\la$ is an $\ell$-tuple $(X_1, \ldots, X_\ell)$  of disjoint sets with each $|X_i|=\la_i$ and $\bigcup_{i=1}^\ell X_i=\{ 1, \ldots, n \}$.
  Note that $\mc P_\la$ has cardinality
    \begin{equation} \label{perm.degree}
  \binom{n}{\la_1,\ldots, \la_\ell}.
  \end{equation}
  The group $S_n$ acts on $\mc P_\la$ in the obvious way, and we obtain the permutation representation $\R[\mc P_\la]$.  This representation space is given by formal linear combinations of elements of $\mc P_\la$, so its degree is given by \eqref{perm.degree}.

For example, if $\la=(1, \ldots, 1) \vdash n$, then $\R[\mc P_\la]$ is the regular representation of $S_n$.  If $\la=(n-1,1)$, then $\R[\mc P_\la]$ is the standard representation $\pi_n$ of $S_n$ on $\R^n$. Note that $S_n$ acts transitively on $\mc P_\la$ with a stabilizer equal to the ``Young subgroup" $S_\la=S_{\la_1} \times \cdots \times S_{\la_\ell}$, so we can also view  $\R[\mc P_\la]$  as the induction from $S_\la$ to $S_n$ of the trivial representation.

The characters of the $\R[\mc P_\la]$, though typically reducible, form an important basis of the representation ring of $S_n$.  See, for example, Section 2.2 of \cite{jameskerber}.

Recall that, for a representation $\pi$ coming from a group action, the character value $\chi_\pi(g)$ is the number of fixed points of $g$.  Write $\Theta_\la$ for the character $\chi_{\R[\mc P_\la]}$.

From this fixed-point principle we compute
\beq
\Theta_\la(s_1)=\sum_{|\la_i| \geq 2} \binom{n-2}{\la_i, \ldots, \la_i-2, \ldots, \la_\ell},
\eeq
since a partition in $\mc P_\la$ is fixed by $s_1$ iff $1$ and $2$ lie in the same part $X_i$ for some $i$.
 Similarly, $ \Theta_\la(s_1s_3)$ equals
 \beq
 \sum_{\substack{1\leq i<j\leq l\\|\lambda_i|\geq 2, |\lambda_j|\geq 2}}{\binom {n-4} {\lambda_1,\dots,\lambda_{i}-2,\dots,\lambda_{j}-2,\dots, \lambda_l}}
+\sum_{|\lambda_k|\geq 4}{\binom {n-4} {\lambda_1,\dots,\lambda_{i}-4,\dots, \lambda_l}}.
\eeq
  This is because a partition in $\mc P_\la$  is fixed by $s_1s_3$ iff either the elements $1,2,3,4$ all lie in the same part $X_i$, or $1,2$ lie in some $X_i$ and $3,4$ lie some other part $X_j$. 
  
  These character values may be used to compute $g_{\R[\mc P_\la]}$ and $h_{\R[\mc P_\la]}$.  (Compare Lemma 17 in \cite{APS}.) For instance if $\la=(1,\ldots, 1) \vdash n$, then $\Theta_\la(s_1)= \Theta_\la(s_1s_3)=0$, so
   $g_{\R[\mc P_\la]}=h_{\R[\mc P_\la]}=\frac{n!}{2}$.  Thus the regular representation of $S_n$ is achiral and spinorial.

The standard representation $\pi_n$ corresponds to $\la=(n-1,1)$.  For this $\la$ we have $\Theta_\la(s_1)=n-2$, so $g_{\R[\mc P_\la]}=1$ and it follows that $\pi_n$ is aspinorial.
   
For easy reference, we collect here results about common representations of $S_n$:
   
  \begin{prop}   For $n \geq 2$ the standard representation $\pi_n$ is achiral and aspinorial, and the sign representation is achiral and aspinorial.
   For $n \geq 4$, the regular representation of $S_n$ is achiral and spinorial.   \end{prop}

\section{Alternating Groups}\label{altgrp}

Now we turn to the alternating group $A_n$, for $n \geq 4$.

\subsection{Spinoriality Criterion}

The group $A_n$ is generated by the permutations
\beq
u_i= s_1 s_{i+1}, \quad (i=1,2,\ldots, n-2)
\eeq
with relations:
\begin{align*}
u_1^3=u_j^2 & =(u_{j-1}u_j)^3=1,\quad (2\leq j\leq n-2),  \\
(u_iu_j)^2 & =1,\quad (1\leq i<j-1,j\leq n-2). \\
\end{align*}
(See for instance \cite{CMGFG}.)

Note that $u_1$ is a $3$-cycle and the other $u_i$ are $(2,2)$-cycles.
 
For a real representation $(\pi,V)$ of $A_n$ again put $h_\pi=\dfrac{\chi_\pi(1)-\chi_\pi(s_1s_3)}{2}$.  Since this is the multiplicity of the eigenvalue $-1$ of $\pi(s_1s_3)$, which has determinant $1$,  the integer $h_\pi$ is  necessarily even.
 
\begin{theorem}\label{spinAn}
A real representation $(\pi,V)$ of $A_n$ is spinorial if and only if  $h_\pi$ is a multiple of $4$.  In this case there is a unique lift.
 
\end{theorem}

\begin{proof}
As in Section \ref{lifting.criteria.subs} we much choose $c_i$ with $\rho(c_i)=\pi(u_i)$ satisfying the same relations as the $u_i$.  Let $c_1$ be a lift of $\pi(u_1)$.  Since $\rho(c_1)^3=\pi(u_1)^3=1$, we have
$c_1^3= \pm 1$.  This determines the sign of $c_1$.

 The $u_j$ and $u_iu_j$ as above, for $j>1$, are all conjugate to $u_2$ in $A_n$.   Therefore all the conditions $c_j^2=1$ and $(c_ic_j)^2=1$ are equivalent to the  condition $c_2^2=1$.  As before, this is equivalent to $h_\pi$ being congruent to $0$ or $3 \mod 4$, but since $h_\pi$ is even, it must be a multiple of $4$. 
 
 Finally, there is a unique choice of signs normalizing $c_2, \ldots, c_{n-2}$ so that 
 \beq
 (c_1c_2)^3=(c_2 c_3)^3= \ldots = 1.
 \eeq
 
 \end{proof}

{\bf Example:} If $\rho$ is the regular representation of $A_n$ (on the group algebra $\R[A_n]$),  then $h_\rho=\dfrac{n!}{4}$, so $\rho$ is spinorial iff  $n \neq 4,5$.

  {\bf Example:} For the standard representation $\pi_n$ of $S_n$, $h_{\pi_n}=2$, so the restriction of $\pi_n$ to $A_n$ is aspinorial.

   \subsection{Real Irreducible Representations}
  
Let us review the relationship between real and complex irreducible representations of a finite group $G$, following \cite{BrokerDieck}.  If $(\pi,V)$ is a complex representation of a group $G$,  write $(\pi_{\R},V_{\R})$ for the realization of $\pi$, meaning that we simply forget the complex structure on $V$ and regard it as a real representation.  If moreover $(\pi,V)$ is an orthogonal complex representation, meaning that it admits a $G$-invariant symmetric nondegenerate bilinear form, then there is a unique real representation $(\pi_0,V_0)$, up to isomorphism, so that $\pi \cong \pi_0 \otimes_{\R} \C$.

It is not hard to see that $\pi_0$ is self-dual iff $\pi$ is self-dual, and that an orthogonal $\pi$ is spinorial, i.e., lifts to $\Spin(V)$, iff $\pi_0$ is spinorial, i.e., lifts to $\Spin(V_0)$.

 
 Every real irreducible representation $\sigma$ of $G$ is either of the form
 \begin{enumerate}
 \item $\sigma=\pi_0$, for an orthogonal irreducible complex representation $\pi$ of $G$, or
 \item $\sigma=\pi_\R$, for an irreducible complex representation $\pi$ of $G$ which is not orthogonal.
 \end{enumerate}

  In the case of $G=S_n$, all complex representations are orthogonal. 
  
    \subsection{Real Irreducible Representations of $A_n$}
  
For a partition $\la$, write $\la'$ for its conjugate partition.  Furthermore write $\epsilon_\la=1$ when the number of cells in the Young diagram of $\la$ above the diagonal is even, and $\epsilon_\la=-1$ when this number is odd.

For example, let $\lambda=(4,3,2,1)$.  Then $\la=\la'$, and there are $4$ cells above the diagonal, shaded in the Young diagram below, so $\epsilon_{\lambda}=1$. 
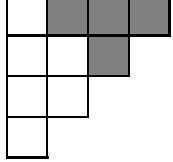
\begin{center}
\begin{ytableau}
*(white) & *(gray)  &*(gray) &*(gray) \\
*(white)  & *(white)  & *(gray) \\
*(white) & *(white)  \\
*(white)
\end{ytableau}
\end{center}

Let $\sigma_\la$ be the (real) Specht module corresponding to $\la$, as before.  Write $\pi_\la=\sigma_\la \otimes \C$ for its complexification, i.e., the complex Specht module corresponding to $\la$.

If $\la \neq \la'$, then $\pi_\la$ restricts irreducibly to $A_n$.
When $\la=\la'$, the restriction of $\pi_\la$ to $A_n$ decomposes into a direct sum of two nonisomorphic representations $\pi_\la^+$ and $\pi_\la^-$.  Either of $\pi_\la^{\pm}$ is the twist of the other by $\sigma_\la(w)$ for any odd permutation $w$.  
The set of $\pi_\la$ with $\la \neq \la'$, together with the $\pi_\la^{\pm}$ for $\la=\la'$, is a complete set of irreducible complex representations of $A_n$.

For $\la=\la'$ we have  \begin{equation} \label{half.chi}
  \chi_{\la}^+(s_1 s_3)=\chi_{\la}^-(s_1 s_3)=\half  \chi_{\la}(s_1 s_3).
  \end{equation}

If moreover $\epsilon_\la=1$, then the representations $\pi_\la^+$ and $\pi_\la^-$ are orthogonal.  We may then define real irreducible representations of $A_n$ by $\sigma_\la^{\pm}=(\pi_\la^{\pm})_0$.
  However when $\epsilon_\la=-1$, the representations $\pi_\la^{\pm}$ are not orthogonal, and therefore the realizations $(\pi_\la^{\pm})_\R$ are irreducible.  Since
  \beq
  \begin{split}
  (\pi_\la^+)_\R \oplus (\pi_\la^-)_\R & \cong (\sigma_\la|_{A_n} \otimes \C)_{\R} \\ 
   & \cong \sigma_\la|_{A_n} \oplus \sigma_\la|_{A_n}, \\
   \end{split}
  \eeq
   we have isomorphisms of real $A_n$-representations:
  \beq
  (\pi_\la^+)_\R \cong (\pi_\la^-)_\R \cong \sigma_\la|_{A_n}.
  \eeq

 From these considerations and Theorem \ref{spinAn}  we conclude:
  
  \begin{theorem} \label{irred.spin.an} A complete list of real irreducible representations of $A_n$ is given by
  \begin{enumerate}
  \item $\sigma_\la|_{A_n}$, where either $\la \neq \la'$, or $\la =\la'$ and $\epsilon_\la=-1$, and
  \item $\sigma_\la^{\pm}$, where $\la=\la'$ and $\epsilon_\la=1$.
  \end{enumerate}
  
  In the first case,  $\sigma_\la|_{A_n}$ is spinorial iff $\chi_\la(s_1s_3) \equiv \chi_\la(1) \mod 8$.  In the second case, $\sigma_\la^+$ is spinorial iff $\sigma_\la^-$ is spinorial iff $\chi_\la(s_1s_3) \equiv \chi_\la(1) \mod 16$.  \end{theorem}
  
{\bf Remark:}  When  $\la=\la'$ the restriction  $\sigma_\la|_{A_n}$ is necessarily spinorial by \eqref{half.chi}, since all $h_\pi$ are even.

 \section{Tables}  \label{table.section} 
  
We illustrate the theory of this paper by means of two tables.   Table 1 below contains the following information for $2 \leq n \leq 6$:

\begin{enumerate}
\item Whether the Specht Module $\sigma_\lambda$ is chiral, i.e., whether $g_\la$ is odd.
\item Whether $\sigma_\lambda$ is spinorial, by Theorem \ref{tfae.spin}.
\item Whether the restriction of $\sigma_\lambda$ to $A_n$ is spinorial, by Theorem \ref{irred.spin.an}.
\end{enumerate}
  
Table 2 below lists for self-conjugate $\la$ with $\epsilon_\lambda=1$, whether the constituents $\sigma_\la^+$ and $\sigma_\la^-$ are spinorial, following Theorem \ref{irred.spin.an}.  This is done for all such $\la$ with $3 \leq |\la| \leq 15$.
 
 
\begin{table}[ht]
\caption{Spinoriality/Chirality of $\sigma_\la$ with $2 \leq |\la| \leq 6$}
\centering

\tabcolsep=.5cm
\begin{tabular}{|p{1cm}|p{1.3cm}|p{1.7cm}|p{1.7cm}|}
\hline
$\lambda$ & Chirality of $\sigma_{\lambda}$  & Spinoriality of $\sigma_{\lambda}$ & Spinoriality of $\sigma_{\lambda}|_{A_n}$  \\
\hline
\multicolumn{4}{|c|}{$|\la|=2$}\\
\hline  
$(2)$      & achiral    & spinorial & spinorial     \\  
\hline  
$(1^2)$      & chiral    & aspinorial & spinorial     \\
\hline
\multicolumn{4}{|c|}{$|\la|=3$}\\
\hline  
$(3)$      & achiral    & spinorial & spinorial     \\  
		\hline  
		$(2,1)$      & chiral    & aspinorial & spinorial      \\
		\hline
		$(1^3)$      & chiral    & aspinorial & spinorial     \\
\hline
		\multicolumn{4}{|c|}{$|\la|=4$}\\
		\hline
		$(4)$      & achiral    & spinorial &  spinorial    \\
		\hline
		$(3,1)$    & chiral    & aspinorial &  aspinorial     \\
		\hline
		$(2,2)$    & chiral    & aspinorial   & spinorial      \\
		\hline
		$(2,1^2)$  & achiral   & aspinorial   & aspinorial       \\
		\hline
		$(1^4)$    & chiral   & aspinorial  & spinorial      \\
		\hline
\multicolumn{4}{|c|}{$|\la|=5$}\\
 \hline
 $(5)$      & achiral    & spinorial & spinorial    \\
 \hline
 $(4,1)$    & chiral    & aspinorial &  aspinorial     \\
 \hline
 $(3,2)$    & achiral    & aspinorial   &  aspinorial      \\
 \hline
 $(3,1^2)$  & chiral   & spinorial   & spinorial       \\
 \hline
 $(2^2,1)$    & chiral   & aspinorial  & aspinorial      \\
 \hline
 $(2,1^3)$    & chiral   & aspinorial  & aspinorial    \\
 \hline
 $(1^5)$    & chiral   & aspinorial  & spinorial       \\
 \hline
\multicolumn{4}{|c|}{$|\la|=6$}\\
 \hline
 $(6)$      & achiral    & spinorial & spinorial     \\
 \hline
 $(5,1)$    & chiral    & aspinorial & aspinorial     \\
 \hline
 $(4,2)$    & chiral    & aspinorial   & spinorial    \\
 \hline
 $(4,1^2)$  & achiral   & aspinorial   & aspinorial     \\
 \hline
 $(3^2)$    & achiral   & aspinorial  & aspinorial      \\
 \hline
 $(3,2,1)$    & achiral   & spinorial  & spinorial       \\
 \hline
 $(3,1^3)$    & achiral   & aspinorial  & aspinorial   \\
 \hline
 $(2^3)$    & chiral   & aspinorial  & aspinorial       \\
 \hline
 $(2^2,1^2)$    & achiral   & aspinorial  & spinorial     \\ 
 \hline
 $(2,1^4)$    & achiral   & aspinorial  & aspinorial     \\ 
 \hline
 $(1^6)$    & chiral   & aspinorial  & spinorial     \\
 \hline
\end{tabular}
\label{exam2}

\end{table}


\newpage

\begin{table}[ht]
	\caption{Spinoriality of $\sigma_\la^{\pm}$ with $\la=\la'$, $\epsilon_\la=1$, and $3 \leq |\la| \leq 15$}
	\centering	

\tabcolsep=.5cm
\begin{tabular}{|c|c|c|}
 \hline
 $\lambda$  &   $|\la|$ &   $\sigma_\la^\pm$     \\
\hline
 $(3,1,1)$   & $5$    &   aspinorial\\  
\hline
$(3,2,1)$      & $6$    &   spinorial \\   
\hline
 $(5,1^4)$    & $9$     &   spinorial \\
\hline
$(5,2,1^3)$ &   $10 $ &    spinorial \\
\hline
$(4,3,2,1)$  &   $10$ &     spinorial\\
\hline
$(4,3,3,1)$ & $11$ & aspinorial \\
\hline
$(7,1^6)$ & $13$ & spinorial \\
\hline
$(7,2,1^5)$ & $14$ & spinorial \\
\hline
$(6,3^2,1^3)$ & $15$ & spinorial \\
\hline
$(5,4,3,2,1)$ & $15$ & spinorial \\
\hline
$(4^3,3)$ & $15$ & spinorial \\
\hline
\end{tabular}
\label{exam3}
\end{table}

  \section{Stiefel-Whitney Classes} \label{SW.section}

  \subsection{Basic Properties}
  Let $G$ be a finite group and $\pi$ a real representation of $G$.  Stiefel-Whitney classes $w_i(\pi)$ are defined for $0 \leq i \leq \deg \pi$ as members of the cohomology groups $H^i(G)=H^i(G,\Z/2\Z)$.  Here $\Z/2\Z$ is trivial as a $G$-module.   
 One considers the total Stiefel-Whitney class in the $\Z/2\Z$-cohomology ring:
  \beq
  w(\pi)=w_0(\pi)+ w_1(\pi)+ \cdots +w_d(\pi) \in H^*(G)=\bigoplus_{i=0}^\infty H^i(G),
  \eeq
  where $d=\deg \pi$.
  
  According to, for example  \cite{GKT}, these characteristic classes satisfy the following properties:
  \begin{enumerate}
  \item $w_0(\pi)=1$.
  \item $w_1(\pi)=\det \pi$, regarded as a linear character in $H^1(G) \cong \Hom(G, \{\pm 1\})$.
  \item If $\pi'$ is another real representation, then $w(\pi \oplus \pi')=w(\pi) \cup w(\pi')$.
  \item \label{incl.w} If $f: G' \to G$ is a  group homomorphism, then $w(\pi \circ f)=f^*(w(\pi))$, where $f^*$ is the induced map on cohomology.
  \item \label{here}  Suppose $\det \pi=1$.  Then $w_2(\pi)=0$ iff $\pi$ is spinorial.
  \end{enumerate}
  
 Note in particular that $w(\chi)=1+\chi$, if $\chi$ is a linear character of $G$.
  
  The last property generalizes as follows:
  \begin{prop} \label{spin.crit.chiral} A real representation $\pi$ is spinorial iff $w_2(\pi)=w_1(\pi) \cup w_1(\pi)$.
  \end{prop}
  
  We will deduce this proposition from the following lemma.
  \begin{lemma} Let $\pi'=\pi \oplus \det \pi$.  Then $\pi$ is spinorial iff $\pi'$ is spinorial.
  \end{lemma}
  
  \begin{proof}
  Let $V'$ be the representation space of $\pi'$; say $V'=V \oplus \R v'$ for some unit vector $v'$ perpendicular to $V$.   Write $\iota: C(V) \to C(V')$ for the canonical injection.  Note that if $x \in C(V)^0$, then $\iota(x)v'=v' \cdot \iota(x)$, and if $x \in C(V)^1$, then $\iota(x)v'=-v' \cdot \iota(x)$.  Write $\rho': \Pin(V') \to O(V')$ for the usual double cover.

   Define $\varphi: \Or(V) \to \SO(V')$ by
 \beq
 \varphi(g)=g \oplus \det(g),
 \eeq
  so that $\pi'=\varphi \circ \pi$.  Write $\Phi_V< \SO(V')$ for the image of $\varphi$.  The essential problem is to construct a lift of $\varphi \circ \rho$.
  Since $\Spin(V)=\Pin(V) \cap C(V)^0$, the map  $\tilde \varphi: \Pin(V) \to \Spin(V')$ defined by
  \beq
\tilde \varphi(x)=\begin{cases} 
&\iota(x),  \text{ if $x \in \Spin(V)$} \\
&\iota(x)v', \text{ if $x \notin \Spin(V)$} \\
\end{cases}
\eeq
is a group homomorphism.  Note that 
\begin{equation} \label{phi.commutes}
\rho' \circ \tilde \varphi=\varphi \circ \rho.   
\end{equation}

Let us see that $\tilde \varphi$ is injective; suppose $\tilde \varphi(x_1)=\tilde \varphi(x_2)$.  Clearly if   $x_1$ and $x_2$ are both in $\Spin(V)$, or both not in $\Spin(V)$, then $x_1=x_2$.  Suppose 
$x_1 \in \Spin(V)$ but $x_2 \notin \Spin(V)$.  Then $v' =\iota(x_2^{-1}x_1)$ and in particular $v' \in \iota(C(V))$.  But this is impossible, say by Corollary 6.7 in Chapter I of \cite{BrokerDieck}.  We conclude that $\tilde \varphi$ is injective.
 
  Write $\tilde \Phi_V< \Spin(V')$ for the image of $\tilde \varphi$; then $\tilde \varphi$ is an isomorphism from $\Pin(V)$ onto $\tilde \Phi_V$ and
\beq
(\rho')^{-1}\Phi_V=\tilde \Phi_V.
\eeq
   
   If $\hat \pi$ is a lift of $\pi$, then $\tilde \varphi \circ \hat \pi$ is a lift of $\pi'$.  Conversely, suppose $\hat \pi'$ is a lift of $\pi'$. Then its image lies in $\tilde \Phi_V$, and therefore $\hat \pi'=\tilde \varphi \circ \hat \pi$ for some homomorphism $\hat \pi: G \to \Pin(V)$.  Since 
   \beq
   \rho' \circ \hat \pi'= \varphi \circ \pi,
   \eeq
   it follows that
   \beq
   \varphi \circ \rho \circ \hat \pi= \varphi \circ \pi.
   \eeq
  Thus $\hat \pi$ is a lift of $\pi$.

  \end{proof}
  
  \begin{proof} (of Proposition  \ref{spin.crit.chiral})
  Note that $\det \pi'=1$, so $\pi$ is spinorial iff $w_2(\pi')=0$.  But 
  \beq
  \begin{split}
  w(\pi') &=w(\pi) \cup w(\det \pi) \\
  &= (1+ \det \pi+w_2(\pi)+ \cdots) \cup (1+ \det \pi) \\
  &=1 + w_2(\pi) + w_1(\pi) \cup w_1(\pi)+ \cdots, \\
  \end{split}
  \eeq
  whence the theorem.
  \end{proof}

 \subsection{The group of order $2$} \label{C2}
  Let $C$ be a cyclic group of order $2$, and write `$\sgn$' for its nontrivial linear character.  Then $H^2(C) \cong \Z/2\Z$; the nonzero element is `$\sgn \cup \sgn$'.  Let $\pi$ be the sum of $m$ copies of the trivial representation with $n$ copies of $\sgn$.  Then 
  \beq
  \begin{split}
  w(\pi) &= w(\sgn) \cup \cdots \cup w(\sgn) \\
  		&= 1+ n \cdot \sgn + \binom{n}{2} \cdot \sgn \cup \sgn  + \cdots. \\
		\end{split}
		\eeq
In particular, $w_2(\pi)= \binom{n}{2} \cdot \sgn \cup \sgn$.  By Proposition \ref{spin.crit.chiral}, $\pi$ is   spinorial iff $n^2 \equiv \binom{n}{2} \mod 2$;
equivalently, $n \equiv 0$ or $3 \mod 4$.

  \subsection{Calculation for $S_n$}

Write 
\beq
e_{\mathrm{cup}}=w_1(\sgn)\cup w_1(\sgn) =w_2(\sgn \oplus \sgn) \in H^2(S_n).
\eeq

 Again write $\pi_n$ for the standard representation of $S_n$ on $\R^n$.
From \cite[Section $1.5$]{Serrespin} we know that $e_{\mathrm{cup}}$ and $w_2(\pi_n)$ comprise a basis for the $\Z/2\Z$-vector space  $H^2(S_n)$.   
 
\begin{prop} The map
\beq
\Phi: H^2(S_n) \to H^2( \lip s_1 \rip) \oplus H^2( \lip s_1 s_3 \rip),
\eeq
given by the two restrictions, is an isomorphism for $n \geq 4$.
\end{prop}
Thus the second $\Z/2\Z$-cohomology of $S_n$ is ``detected" by these cyclic subgroups; compare Corollary   \ref{rest.spin.prop} above and Theorem VI.1.2 in \cite{Milgram}.

\begin{proof} 
Since $\Phi$ is a linear map between $2$-dimensional $\Z/2\Z$-vector spaces, it suffices to prove that its rank is  $2$.
Let $b_1$ be the generator of $H^2( \lip s_1 \rip)$, and $b_2$ be the generator of $H^2( \lip s_1 s_3 \rip)$.

The restriction of $\pi_n$ to $\lip s_1 \rip$ decomposes into a trivial $(n-1)$-dimensional representation plus one copy of $\sgn$.   The restriction to $\lip s_1 s_3 \rip$ contains two copies of $\sgn$.  Therefore $\Phi(w_2(\pi_n))=(0,b_2)$.  Similarly $\Phi(w_2(\sgn \oplus \sgn))=   \Phi(e_{\mathrm{cup}})=      (b_1,0)$.  Thus $\Phi$ has rank $2$, as required.

\end{proof}

\begin{theorem}  \label{seven.point.two} 
	For $\pi$ a real representation of $S_n$,  with $n \geq 4$, we have
	\beq
\begin{split}
	w_2(\pi) &=\left[\frac{g_\pi}{2}\right]e_{\mathrm{cup}}+\frac{h_\pi}{2}w_2(\pi_n) \\
	 &=\left[\frac{\chi_V(1)-\chi_V(s_1)}{4}\right]e_{\mathrm{cup}}+\frac{\chi_{V}(1)-\chi_V(s_1s_3)}{4}w_2(\pi_n). \\
	 \end{split}
	 \eeq
\end{theorem}
Here $[ \cdot  ]$ denotes the greatest integer function.
  
\begin{proof} Suppose first that $\pi$ is achiral.  Since $e_{\mathrm{cup}}$ and $w_2(\pi_n)$ form a basis of $H^2(S_n)$ we must have
 \beq
 w_2(\pi) =c_1 e_{\mathrm{cup}}+c_2 w_2(\pi_n),
 \eeq
for some $c_1,c_2 \in \Z/2\Z$.  Thus $\Phi(w_2(\pi))=c_1b_1+c_2b_2$.
By the Stiefel-Whitney class properties (\ref{incl.w})  and (\ref{here}),
\beq
c_1=0 \Leftrightarrow \pi|_{\lip s_1 \rip} \text{ is spinorial } \Leftrightarrow 4 | g_\pi
\eeq
and
\beq
 c_2=0 \Leftrightarrow \pi|_{\lip s_1s_3 \rip} \text{ is spinorial } \Leftrightarrow 4 | h_\pi.
\eeq

Thus $c_1 \equiv \frac{g_\pi}{2}  \mod 2$ and $c_2 \equiv \frac{h_\pi}{2}  \mod 2$.

 If $\pi$ is chiral, then $\pi'=\pi \oplus \sgn$ is achiral.  From the identity $w_2(\pi)=w_2(\pi')+e_{\mathrm{cup}}$, we deduce that
  \beq
 w_2(\pi) =\frac{g_\pi-1}{2} e_{\mathrm{cup}}+\frac{h_\pi}{2}w_2(\pi_n).
 \eeq
\end{proof}

{\bf Remark:} For $n=2,3$, similar reasoning gives $H^2(S_n) \cong H^2(\lip s_1 \rip)$ and
\beq
	w_2(\pi) =\left[\frac{g_\pi}{2}\right]e_{\mathrm{cup}} =\left[\frac{\chi_V(1)-\chi_V(s_1)}{4}\right]e_{\mathrm{cup}}.
\eeq

{\bf Remark:} Since the groups $H^2(A_n,\Z/2\Z)$ have order $1$ or $2$,  computing the Stiefel-Whitney class of a real representation of $A_n$ is equivalent to determining its spinoriality, which we have already done.

\section{Products} \label{products.section}

Spinoriality for representations of $S_n \times S_{n'}$ can also be determined by means of generators and relations.  (See  Theorem 5.4.1 in \cite{Jyotirmoy}.)  However we will instead obtain a satisfactory criterion by simply feeding our calculation of $w_2(\pi)$ into the machinery of Stiefel-Whitney classes.

  \subsection{External tensor products}  \label{six.point.two}
  
  Let $G,G'$ be finite groups, let $(\pi,V)$ be a real representation of $G$, and let $(\pi',V')$ be a real representation of $G'$.  Write $\pi \boxtimes \pi'$ for the external tensor product representation of $G \times G'$ on $V \otimes V'$.  One computes
  \beq
  \det(\pi \boxtimes \pi')=\det(\pi)^{\deg \pi'} \cdot \det(\pi')^{\deg \pi},
  \eeq
  and hence
  \beq
  w_1(\pi \boxtimes \pi')=\deg \pi' \cdot w_1(\pi)+ \deg \pi \cdot w_1(\pi'),
  \eeq
which is  an element of
  
\beq
  H^1(G \times G') \cong H^1(G) \oplus H^1(G').
  \eeq
  
  The famous ``splitting principle" (e.g., proceeding as in Problem 7-C of \cite{mch}) similarly gives
  \beq
  \begin{split}
  w_2(\pi \boxtimes \pi') & = \deg \pi' \cdot w_2(\pi)+ \binom{\deg \pi'}{2} w_1(\pi) \cup w_1(\pi) + (\deg \pi \deg \pi'-1) w_1(\pi)   \otimes w_1(\pi') \\
  &+ \binom{\deg \pi}{2} w_1(\pi') \cup w_1(\pi')+ \deg \pi \cdot w_2(\pi'),\\
  \end{split}
  \eeq
  as an element of 
  \beq
  H^2(G \times G') \cong H^2(G) \oplus \left( H^1(G) \otimes H^1(G') \right) \oplus H^2(G').
  \eeq

   Finally, $w_2(\pi \boxtimes \pi') + w_1(\pi \boxtimes \pi') \cup w_1(\pi \boxtimes \pi')$ comes out to be
 \beq
 \begin{split}
 \deg \pi' \cdot w_2(\pi) &+ \binom{\deg \pi'+1}{2} w_1(\pi) \cup w_1(\pi) + (\deg \pi \deg \pi' +1)w_1(\pi) \otimes w_1(\pi')   \\
 &+ \binom{\dim \pi+1}{2} w_1(\pi') \cup w_1(\pi')+ \deg \pi \cdot w_2(\pi').\\
 \end{split}
 \eeq
 Thus $\pi \boxtimes \pi'$ is spinorial (by Proposition \ref{spin.crit.chiral}) iff  all of the following vanish:
\begin{enumerate}
\item  $\deg \pi' \cdot w_2(\pi)+ \binom{\deg \pi'+1}{2} w_1(\pi) \cup w_1(\pi)$,
\item $(\deg \pi \deg \pi' +1)w_1(\pi) \otimes w_1(\pi')$, and
\item $ \binom{\deg \pi +1}{2} w_1(\pi') \cup w_1(\pi')+ \deg \pi \cdot w_2(\pi')$.
\end{enumerate}

\subsection{Products of Symmetric Groups}

We now prove   Theorem \ref{prod.thm.intro}.  Let $\pi,\pi'$ be representations of $S_n$ and $S_{n'}$.   Write $f=f_\pi$, $f'=f_{\pi'}$ and similarly for $g,h,g'$ and $h'$. Let $\Pi=\pi \boxtimes \pi'$; all representations of $S_n \times S_{n'}$ are sums of such representations.

\begin{proof}
 
 By Proposition \ref{spin.crit.chiral},  $\Pi$ is spinorial iff 
 \beq
 w_2(\Pi) = w_1(\Pi) \cup w_1(\Pi).
 \eeq
 From Theorem \ref{seven.point.two} and (1)-(3) of Section \ref{six.point.two} we deduce that $\Pi$ is spinorial iff all of the following are even:
 \begin{enumerate}
  \item $f' \cdot \frac{h}{2}$,
 \item $f' \left[\dfrac{g}{2} \right]+ \binom{f'+1}{2} g$,

 \item $(ff'+1) gg'$,
   \item $f \cdot \frac{h'}{2}$, and
 \item $f \left[\dfrac{g'}{2} \right]+ \binom{f+1}{2} g'$.

\end{enumerate}
  
 Note that if $\Pi$ is spinorial, then its restriction to $S_n \times \{ 1 \}$, which amounts to $f'$ copies of $\pi$, is spinorial.  From before, this implies that   $f' \cdot \frac{h}{2}$ is even, and
 $f'g$ is congruent to $0$ or $3 \mod 4$.  One can verify this $f'g$ condition is equivalent to (2) being even.  Thus (1),(2),(4), and (5) above are all even iff the restrictions of $\Pi$ to $S_n \times \{ 1 \}$ and $\{ 1 \} \times S_{n'}$ are spinorial.   Theorem \ref{prod.thm.intro} follows from this, since $ff'=\deg \Pi$.
   
\end{proof}

 \bibliographystyle{alpha}
 
\bibliography{mybib}

  \end{document}